\documentclass[11pt,a4paper]{amsart}
\usepackage[utf8]{inputenc}
\usepackage{lmodern}
\usepackage{booktabs}
\usepackage[final]{hyperref}
\usepackage{esint, amsmath,amssymb,amsthm}
\usepackage[abbrev]{amsrefs}
\usepackage{a4wide}
\usepackage{tocvsec2}
\usepackage{enumerate}
\usepackage{wasysym}
\usepackage{multirow}
\usepackage{bigints}
\usepackage{color}
\usepackage[showonlyrefs]{mathtools}

\newcommand\ddfrac[2]{\frac{\displaystyle #1}{\displaystyle #2}}

\newtheorem{thm}{Theorem}[section]
\newtheorem{lem}[thm]{Lemma}
\newtheorem{prop}[thm]{Proposition}

\theoremstyle{definition}

\theoremstyle{remark}

\newtheorem{remark}[thm]{Remark}
\numberwithin{equation}{section}

\newcommand{\R}{{\mathbb R}}

\newcommand{\dif}{\,\mathrm{d}}

\allowdisplaybreaks

\begin{document}

\title[Symmetry breaking of groundstates] {Quantitative symmetry breaking of groundstates for a class of weighted Emden-Fowler equations}

\author[C.\ Mercuri]{Carlo Mercuri}
\address[C.\ Mercuri]{Swansea University\\ Department of Mathematics\\ Fabian Way, Crymlyn Burrows, SA1~8 EN, Skewen, 
Swansea, Wales, United Kingdom}
\email{C.Mercuri@swansea.ac.uk}

\author[E.\ Moreira dos Santos]{Ederson Moreira dos Santos}
\address[E.\ Moreira dos Santos]{ Universidade de S\~ao Paulo\\ Instituto de Ci\^encias Matem\'aticas e de Computa\c{c}\~ao
 \\CEP 13566-590\\S\~ao Carlos - SP - Brazil}
\email{ederson@icmc.usp.br}




\begin{abstract} We consider a class of weighted Emden-Fowler equations

\begin{equation} \tag{$\mathcal P_{\alpha}$} \label{eqab}
\left\{\begin{array}{ll}
-\Delta u=V_{\alpha} (x) \, u^p & \text{in} \,\,B,\\
u>0 & \text{in} \,\,B,\\
u=0 &  \text{on}\,\,\partial B,
\end{array}\right.
	\end{equation}
posed on  the unit ball $B=B(0,1)\subset \R^N$, $N \geq1$. We prove that symmetry breaking occurs for the groundstate solutions as the parameter $\alpha \rightarrow \infty.$
The above problem reads as a possibly large perturbation of the classical H\'enon equation. We consider a radial function $V_\alpha$ having a spherical shell of zeroes at $|x|=R \in (0,1].$ For $N \geq 3$, a quantitative condition on $R$ for this phenomenon to occur is given by means of universal constants, such as the best constant for the subcritical Sobolev's embedding $H^1_0(B)\subset L^{p+1}(B).$ In the case $N=2$ we highlight a similar phenomenon when $R=R(\alpha)$ is a function with a suitable decay. Moreover, combining energy estimates and Liouville type theorems we study some qualitative and quantitative properties of the groundstate solutions to (\ref{eqab}) as $\alpha \rightarrow \infty.$ \end{abstract}

\maketitle

\section{Introduction and results}
\settocdepth{section}

 For $R\in (0,1)$ and $\alpha>0$ we set
\begin{equation} \tag{$V$} \label{def-h}
V(r)=
\left\{\begin{array}{ll}
\left(1-\frac{r}{R}\right)^{\alpha} & \text{ if } 0\leq r< R, \smallskip\\
 \left(1-\frac{1-r}{1-R}\right)^\alpha &\text{ if } R\leq r\leq 1. \smallskip
\end{array}\right.
\end{equation}
We also set $V(|x|) = |x|^{\alpha}$ for $R= 0$ and $V(|x|) = (1- |x|)^{\alpha}$ for $R=1$. We will use $R$-subscripts and $\alpha$-subscripts in those instances where the dependence on $R$ and $\alpha$ plays any~role.
We consider the boundary value problem
\begin{equation} \tag{$\mathcal P$} \label{eq}
\left\{\begin{array}{ll}
-\Delta u=V(|x|) \, |u|^{p-1}u & \text{in} \,\,B,\\
u=0 &  \text{on}\,\,\partial B,
\end{array}\right.
	\end{equation}
on the unit ball $B=B(0,1)\subset \R^N$. Throughout the paper we shall consider $p$ superlinear and Sobolev-subcritical, namely $p\in\left(1,\frac{N+2}{N-2}\right)$ if $N \geq 3$ and $p>1$ if $N=1,2$.	

Note that (\ref{eq}) interpolates two opposite situations: with $R=1$  we obtain 
\begin{equation} \tag{$\mathcal{D}$} \label{eqGNN}
\left\{\begin{array}{ll}
-\Delta u=(1-|x|)^\alpha\, |u|^{p-1}u & \text{in} \,\,B,\\
u=0 &  \text{on}\,\,\partial B,
\end{array}\right.
	\end{equation}
whose positive solutions, in particular groundstates, are radially symmetric by the celebrated result of Gidas-Ni-Nirenberg \cite{Gidas}, whereas with $R= 0$ we are led to the H\'enon equation \cite{Henon}
\begin{equation} \tag{$\mathcal{H}$} \label{classH}
\left\{\begin{array}{ll}
-\Delta u=|x|^\alpha\, |u|^{p-1}u & \text{in} \,\,B,\\
u=0 &  \text{on}\,\,\partial B,
\end{array}\right.
	\end{equation}
whose groundstate solutions (which we know are positive) \cite{Smets, BW} are nonradial for large enough values of $\alpha$. Here we highlight, quantitatively, a stability phenomenon for the symmetry breaking of groundstate solutions. In fact, we find a condition on $R$, by means of universal constants, for the symmetry breaking of groundstate solutions of (\ref{eq}) to occur for any $\alpha$ suitably large. In the present context $N$ and $p$ are parameters which are fixed from the start. 
We recall that if $u$ is a groundstate solution of \eqref{eq}, with $\alpha\geq0$, then $|u|$ is also a groundstate solution. Then, by the strong maximum principle, $|u|>0$ in $B$ and hence either $u>0$ in $B$ or $u<0$ in $B$. So, throughout this paper, we will work with positive groundstate solutions.
We know from \cite[Theorem 2.6]{BWW} that any groundstate solution of \eqref{eq} is foliated Schwarz symmetric if $R\in [0,1)$ and radially symmetric and radially decreasing if $R=1$. Hence, we may assume that the maximum of any groundstate solution of \eqref{eq} is attained at a point of the form $(s_{\alpha}, 0, \ldots, 0) \in B$ where $s_{\alpha} \in [0,1)$.

Non-autonomous equations like (\ref{eq}) naturally arise in Astrophysics when describing the stability of stellar structures. 
One of the most remarkable predictions in General Relativity, see e.g. \cite{Wald} and \cite{Haw}, is the existence of black holes. These were known just as a mathematical entity since Schwarzschild's  stationary and spherically symmetric singular solution to Einstein's equations, until recent years when in fact black hole like objects have been perceived \cite{existBHGC}, and remarkably, photographed \cite{Pic-blackhole}. 

In the last few decades astrophysicists and mathematicians have devoted intensive effort to detect and understand the existence of black holes in globular clusters. In 1972, Peebles \cite{peebles1,peebles2} published seminal works describing a stationary distribution of stars near a massive collapsed object, such as a black hole, located at the cluster centre.

As shown in \cite{batt,batt-f-horst,li,li-santanilla,batt-li}, the admissibility of stationary and spherically symmetric stellar models, whose dynamics is based on coupling Vlasov and Poisson equations, is equivalent to the solvability of an Emden-Fowler type equation
\[
-\Delta U(x)  = f(|x|, U(x)) \quad \text{in} \quad \R^3.
\]
These equations are extensively considered in Astrophysics as a class of nonlinear Poisson equations describing certain self-gravitating, spherically symmetric stellar systems: the well-known Lane-Emden equation, written in PDE form as
\[
-\Delta U(x)  =|U(x)|^{p-2}U(x) \quad \text{in} \quad \R^3,
\]
describes stellar systems modelled as {\it polytropic} fluids, namely those where a nonlinear (polynomial) relation between the mass density, related to a certain power of $U,$ and the pressure holds; see e.g. the classical book of the 1983 Nobel in Physics laureate S. Chandrasekhar \cite{Stellar}. 
The particular case $f(|x|, U(x)) = |x|^{\alpha}|U(x)|^{p-2}U(x)$, $\alpha>0$, $p>2$, namely
\[
-\Delta U(x)  = |x|^{\alpha}|U(x)|^{p-2}U(x) \quad \text{in} \quad \R^3,
\]
has been proposed in 1973 by H\'enon \cite{Henon} in the context of the so-called concentric shell model, as a correction/generalisation to the Lane-Emden equation, to investigate numerically the stability of spherical steady state stellar systems.

When considering (\ref{classH}), namely the Dirichlet problem associated with the H\'enon equation, in \cite{Smets} it has been shown that at least two positive solutions exist for large values of $\alpha.$ Namely, a radial solution, and surprisingly a least energy solution which is not radially symmetric. Moreover in \cite{CaoPeng,BW} it has been shown that the presence of $|x|^\alpha,$ as $\alpha$ gets large, forces the least energy solutions to concentrate as far as possible from the origin (in fact near the boundary); see also \cite{pacella-indiana} and references therein. In the case of (\ref{eq}) for $R\in (0,1]$ in line with these results, one may naturally expect that, as $\alpha$ gets large, concentration should occur far away from the shell of zeroes of $V,$ $\{x\,:\, |x|=R\},$ namely either at the origin or at the boundary.

\subsection*{Notation.}  $B_\rho=B(0,\rho) \subset \R^N$ is the ball centred at the origin with radius $\rho>0.$ We set $B=B_1,$ whose boundary has $(N-1)$-dimensional Lebesgue measure given in terms of the classical Gamma-function by 
\[
|\partial B|=|\mathbb S^{N-1}|= \frac{2\pi^{N/2}}{\Gamma(N/2)}.
\]
Solutions to \eqref{eq} are critical points of
\begin{equation}\tag{$\mathcal I$} \label{Energy}
I(u)= \frac{1}{2}\int_{B} |D u|^2 \dif x-\frac{1}{p+1}\int_B V(|x|)|u|^{p+1}\dif x, \quad u \in H^1_0(B),
\end{equation}
and its groundstates can be found by suitable rescaling of optimisers to the best constant
\begin{equation}\tag{$\mathcal S_{\alpha}$}\label{S-nonradial}
S_{\alpha}= \inf_{0\neq u\in H^1_0(B)} \ddfrac{\int_B |D u|^2 \dif x}{\left(\int_B V(|x|)|u|^{p+1}\dif x\right)^{\frac{2}{p+1}}},
\end{equation}
and its least energy radial solutions correspond to rescaling of optimisers to the constant
\begin{equation}\tag{$\mathcal S_{\alpha,\textrm{rad}}$} \label{S-radial}
S_{\alpha,\textrm{rad}}= \inf_{0\neq u\in H^1_{0,\textrm{rad}}(B)} \ddfrac{\int_B |D u|^2 \dif x}{\left(\int_BV(|x|)|u|^{p+1}\dif x\right)^{\frac{2}{p+1}}},
\end{equation}
which are achieved by standard direct minimisation in $H^1_0(B)$ and, respectively in $H^1_{0,\textrm{rad}}(B).$
In particular, we consider a positive optimiser $u^*_{\alpha,\textrm{rad}}$ to $S_{\alpha,\textrm{rad}}$ such that $$\int_B V(|x|)|u^*_{\alpha,\textrm{rad}}|^{p+1}\dif x=1.$$ Then by the Palais symmetric criticality principle
\begin{equation} \label{eq minimiser}
-\Delta u^*_{\alpha,\textrm{rad}}=S_{\alpha,\textrm{rad}} V(|x|) \, (u^*_{\alpha,\textrm{rad}})^p
\end{equation}
holds weakly in $H^1_0(B).$ Since the rescaled function $ u_{\alpha,\textrm{rad}}=\left(S_{\alpha,\textrm{rad}}\right)^{\frac{1}{p-1}}u^*_{\alpha,\textrm{rad}}$ solves
\begin{equation} \label{eq rescaled minimiser}
-\Delta u_{\alpha,\textrm{rad}}=V(|x|) \, (u_{\alpha,\textrm{rad}})^p
\end{equation}
weakly in $H^1_0(B),$
it is easy to see that the relation between the energy level
\begin{equation}\tag{$\mathcal C_{\alpha, \textrm{rad}}$} \label{least energy}
C_{\alpha, \textrm{rad}}= I_\alpha(u_{\alpha,\textrm{rad}})
\end{equation}
and $S_{\alpha,\textrm{rad}}$ is
\begin{equation} \label{relation}
S_{\alpha,\textrm{rad}}=\left(\frac{2(p+1)}{p-1}\right)^{\frac{p-1}{p+1}} \left(C_{\alpha, \textrm{rad}}\right)^{\frac{p-1}{p+1}}.
\end{equation}
We will extensively refer to this relation, as well as to such a positive radial solution $u_{\alpha,\textrm{rad}}$ to (\ref {eq}). 
The same scaling argument using a positive optimiser $u^*_\alpha$ to $S_\alpha$ allows one to define $u_\alpha=S_\alpha ^{\frac{1}{p-1}}u^*_\alpha.$ We refer to such $u_\alpha$ as a {\it groundstate solution}  to (\ref{eq}).  It is well-known that, among all the nontrivial critical points of $I_\alpha,$ $u_\alpha$ has least energy, and that $C_{\alpha}= I_\alpha(u_{\alpha})$ is in particular a mountain pass level for $I_\alpha.$ Finally we set
\begin{equation}\tag{$\mathcal S$} \label{Sobolev}
S= \inf_{0\neq u\in H^1_0(B)} \ddfrac{\int_B |D u|^2 \dif x}{\left(\int_B |u|^{p+1}\dif x\right)^{\frac{2}{p+1}}},
\end{equation}
i.e. the best constant for the subcritical Sobolev's embedding.

\subsection*{Asymptotic notation.} Throughout the paper we use the following asymptotic notation for $\alpha\rightarrow +\infty$. For real valued functions $f(\alpha),g(\alpha)\geq0$ defined for $\alpha$ large we write:
\begin{itemize}
\item  $f(\alpha)=o(g(\alpha))$ as $\alpha\to+\infty$ if $g(\alpha)\neq0$ for $\alpha$ large and $\lim_{\alpha\rightarrow+\infty} \frac{f(\alpha)}{g(\alpha)}=0$;
\item $f(\alpha)=O(g(\alpha))$ as $\alpha\rightarrow+\infty$ if there exists $C>0$ such that $f(\alpha)\leq C g(\alpha)$ for $\alpha$ large;
\item $f(\alpha)\asymp g(\alpha)$ if it holds that $f(\alpha)=O(g(\alpha))$ and $g(\alpha)=O(f(\alpha));$
\item  $f(\alpha)\sim g(\alpha)$ if $g(\alpha)\neq0$ and $\lim_{\alpha\rightarrow+\infty} \frac{f(\alpha)}{g(\alpha)}=1$.
\end{itemize}

\subsection*{Main results}

Using the convergence result in Proposition \ref{prop:limit} below and the symmetry breaking result for groundstate solutions for the H\'enon equation \cite{Smets} , we obtain the following preliminary observation.
\begin{prop}[An embryonic symmetry breaking result]\label{GNNH} Let $N \geq 1$ and $\alpha^*>0$ be such that no groundstate of the H\'enon equation \eqref{classH} is radially symmetric for all $\alpha> \alpha^*$. Then, given any $\alpha > \alpha^*$, there exists $R_0 = R_0(\alpha) \leq 1$ such that no groundstate of \eqref{eq} is radially symmetric if $0\leq R < R_0$.
\end{prop}

One of our main results is a quantitative estimate for a radius $R_0$ which is sufficient for symmetry breaking to occur for {\it all} the groundstate solutions relative to $R<R_0.$  The price we pay for such a quantitative bound on $R$ is a {\it non-uniform} symmetry breaking result with respect to $\alpha.$ More precisely, we prove the following result.

\begin{thm}[Quantitative symmetry breaking]  \label{symmetry breaking theorem}
Let $N\geq 3$, $1<p<\frac{N+2}{N-2}$ and set 
\begin{equation}\label{Kappa}
K(N,p)= (N-2)^{-\frac{p+1}{2}} |\mathbb S ^{N-1}|^{\frac{1-p}{2}}\Gamma\left(N-\frac{(p+1)(N-2)}{2}\right) .
\end{equation}
For any $R \in [0,1)$ with 
\begin{equation}\label{condition}
R^{\frac{2N}{p+1} - (N-2)}\min\left\{e^{\frac{4}{R(p+1)}}, e^{\frac{4}{(1-R)(p+1)}}\right\}<  K(N,p)^{-\frac{2}{p+1}} S^{-1},
\end{equation}
there exists $\alpha(R)$ such that
$$S_\alpha<S_{\alpha,\textrm{rad}}$$
and groundstates of \eqref{eq} are nonradial for all $\alpha>\alpha(R).$
\end{thm}
\begin{remark}
Since $\min\left\{e^{\frac{4}{R(p+1)}}, e^{\frac{4}{(1-R)(p+1)}}\right\}\leq e^{\frac{8}{p+1}}$ for $R\in(0,1),$ (\ref{condition}) yields a radius $$R_0:=\Big [e^{-\frac{8}{p+1}}K(N,p)^{-\frac{2}{p+1}} S^{-1}\Big]^{\frac{p+1}{2N-(p+1)(N-2)}}$$ such that for any $R<R_0,$ a value $\alpha(R)$ exists such that
 groundstates of \eqref{eq} are nonradial for all $\alpha>\alpha(R).$
\end{remark}

\begin{remark}
We stress that a symmetry breaking result, uniform with respect to $R \in [0,1)$, based on the growth  of $S_{\alpha}$ and $S_{\alpha, rad}$, cannot be obtained. Indeed, this follows from the convergence \eqref{limitconstants}, combined with
\[
S_{\alpha} = S_{\alpha, rad} \asymp \alpha^{\frac{2N}{p+1}-(N-2)} \ \ \text{as} \ \ \alpha \to \infty, \ \ \text{for $R=1$},
\]
which is a consequence of \eqref{eq:boundconstant} and \eqref{asymptotic3}, and 
\[
S_{\alpha}  \asymp \alpha^{\frac{2N}{p+1}-(N-2)} \ \ \text{and} \ \ S_{\alpha, rad} \asymp \alpha^{1+ \frac{2}{p+1}} \ \ \text{as} \ \ \alpha \to \infty \ \ \text{for $R=0$},
\]
which have been proved in \cite[eq. (5)]{Smets} and \cite[Proposition 3.1]{CaoPeng}.
\end{remark}

Theorem \ref{symmetry breaking theorem} is achieved by the method which we may outline as follows. We first deduce an asymptotic upper bound on $S_\alpha$ in Section \ref{Slower}. Then we obtain asymptotic lower bound estimates on $S_{\alpha,\textrm{rad}}$ by means of intermediate inequalities developed in the main body of the paper, Section \ref{main body}. These are obtained by Nehari's and Pohozaev's identities, expressing integrals of radial functions in terms of Euler's Beta functions (involving $\alpha$) after using Ni's inequality \cite{Ni}; see also \cite{Strauss}. In performing this we pay attention to the contribution of the origin to integrals of the form 
$\displaystyle{ \int_{B_{R}}}\,\, \ldots\,\, \dif x,$
which are the technical obstruction for symmetry breaking to occur for all $R<1.$ This is expected, as we have recalled that symmetry holds when $R=1.$  The aforementioned intermediate estimates take into account possibly different asymptotic contributions, namely $A_\alpha(R)$ and $B_\alpha$ in Lemma \ref{lemma3}, which may be regarded as a weighted influence produced by both the (\ref{eqGNN}) and the (\ref{classH}) limiting PDEs; see also Remark \ref{lowerSbound}.  In fact the proof of Theorem \ref{symmetry breaking theorem} in Section \ref{symmetry breaking proof} consists in analysing the (least favourable) case where the lower bound for $S_{\alpha,\textrm{rad}}$ is asymptotically equivalent to the upper bound for $S_{\alpha}.$ This corresponds to $A_\alpha(R)$ being asymptotically strong: a phenomenon which does not occur when dealing with the H\'enon equation \cite{Smets}. In this case the condition $R<R_0$ is found comparing the constants so that $S_{\alpha,\textrm{rad}}>S_{\alpha}$ holds asymptotically and strictly. We believe that this analysis relates in a nontrivial way to that for the classical H\'enon equation as given in \cite{Smets}. As a byproduct of this method we obtain a lower bound on the subcritical best constant $S$; see also Remark \ref{lowerSbound}\,.

\begin{remark} [A lower bound for $S$]\label{lowerCorollary}
Let $N\geq 3$, $S=S(N,p)$ be the best constant defined in (\ref{Sobolev}) and $K(N,p)$ as in \eqref{Kappa}. Then by \cite{Gidas} no symmetry breaking occurs for $R=1$ and so
\begin{equation}\label{Sobolev estimate}
(N-2) |\mathbb S ^{N-1}|^{\frac{p-1}{p+1}} \left(\Gamma\left(N-\frac{(p+1)(N-2)}{2}\right)\right)^{-\frac{2}{p+1}} e^{-\frac{4}{p+1}}\leq S(N,p).
\end{equation}
In a recent paper \cite{Anello} it has been proved that the function $$p\mapsto S(N,p)|\mathbb S^{N-1}|^{2/(p+1)}$$ is decreasing on $(0,2^*-1),$ for $N\geq 3$ and on $(0,\infty)$ for $N=1,2.$ Using this result one has a lower bound by means of the explicitly known value of the classical Sobolev constant for $p=2^*-1$ \cite{Talenti}. We wonder if our method would yield an improved lower bound with a possibly different choice of $V.$
\end{remark}

The case $N=2$ is more tricky. In this case we use an {\it a priori} logarithmic radial estimate    
from \cite[Lemma 2.5]{BonheureSerraTarallo} which is the low dimensional analog of Ni's classical estimate in \cite{Ni}. Following the scheme of the proof of Theorem \ref{symmetry breaking theorem}  for $N=2$ one finds that, as a consequence of the `$\log$' factor, the  asymptotic bounds for $S_{\alpha}$ and $S_{\alpha,\textrm{rad}}$ do not match, see Remark \ref{remarkN2}. Nevertheless, allowing the radius $R$ to be a suitable function of $\alpha$ yields the following. 

\begin{thm}[`Moving shell' symmetry breaking]  \label{symmetry breaking theorem2}
Let $N= 2$ and $p>1.$ Let $1>R(\cdot)\geq 0$ be a function such that for some $\delta>0$
\begin{equation}\label{N2hyp}
\limsup_{\alpha\rightarrow \infty} \alpha^\delta R(\alpha)<\infty.
\end{equation}
Then, there exists $\alpha_*>0$ such that
$$S_\alpha<S_{\alpha,\textrm{rad}}$$
and groundstates of \eqref{eq} are nonradial for all $\alpha>\alpha_*.$
\end{thm}

The bound provided by Lemma \ref{Salpha} below is also used to achieve the following qualitative result, which essentially states that under a certain growth assumption on the picks of the groundstates, as $\alpha$ gets large, maximum points are allowed to accumulate only at the \linebreak boundary or at the origin, as heuristically pointed out earlier in this introduction. The proof also combines in an essential way a blow-up argument with Liouville theorems, such as Gidas-Spruck's \cite{Gidas0}. 

\begin{thm} [Boundary or origin concentration] \label{nec} Let $N\geq 3$ and $R\in(0,1)$ be fixed. Given $\alpha>0$, let $u_{\alpha}$ be a groundstate solution of \eqref{eq}. Set $\beta_{\alpha}:=\max_{x\in \overline{B}} u_{\alpha}=u_{\alpha}(x_{\alpha})$, where $x_{\alpha}= (s_{\alpha},0, \ldots, 0)$ with $s_{\alpha}\in [0,1)$. Assume that 
\begin{equation}\label{necessarygrowth}
\liminf_{\alpha\rightarrow \infty}\,\alpha^{\frac{2}{p-1}}\beta_{\alpha}^{-1}<\infty.
\end{equation} Then $(x_{\alpha})$ has at most two accumulation points, namely either the origin or $(1,0,\ldots,0)$.
\end{thm}

Since for the H\'enon equation, namely for $R=0,$ it holds that $\beta_{\alpha} \asymp \alpha^{2/(p-1)}$ (see e.g. \cite{BW} and \cite{CaoPeng}) and therefore 
$$
\int_{\R^N} |D v_{\alpha}|^2\dif x=O(1),
$$
with $v_{\alpha}(x):=\frac{1}{\beta_{\alpha}} u_{\alpha} (x_{\alpha}+\beta_{\alpha}^{\frac{1-p}{2}}x),$ it is reasonable to believe that (\ref{necessarygrowth}) holds for groundstates when $R\in(0,1)$.  In fact, as we show in the following proposition, (\ref{necessarygrowth}) holds for groundstates also in the case of the endpoint $R=1.$ The proof is based on the estimates developed in the main body of the paper, Section \ref{main body}, together with the symmetry result of Gidas-Ni-Nirenberg \cite{Gidas}; see also Remark \ref{lowerSbound}.
 
 \begin{prop}\label{R1}
Let $N\geq 3$. Given $\alpha>0$, let $u_{\alpha}$ be a groundstate solution of \eqref{eqGNN}. Set $\beta_{\alpha}:=\max_{x\in \overline{B}} u_{\alpha}$. Then, as $\alpha \to \infty,$ it holds that $\beta_{\alpha} \asymp\alpha^{2/(p-1)}$.
	 \end{prop}

\begin{remark}
Based on the proof of Lemma \ref{Salpha}  it is natural to expect that when $R<1/2$ concentration holds at the boundary, whereas when $R>1/2$ concentration may occur at the origin. In the case $R=1/2$ a groundstate may find equally convenient to concentrate around either the origin, the boundary, or both: in principle one may have different maximum points $(x_{\alpha}),(x'_{\alpha})$ which converge respectively to the origin and to $(1,0,\ldots, 0)$.  
\end{remark}

\section{Estimate for $S_\alpha$}\label{Slower}
Here we write $S_{\alpha}$ as $S_{\alpha,R}$ to emphasize its dependence on $R$.
\begin{lem}\label{Salpha}
Let $N\geq 1$. Given any $R \in [0,1]$,
\begin{equation}\label{eq:boundconstant}
\limsup_{\alpha\rightarrow \infty} S_{\alpha, R} \, \alpha^{N-2-\frac{2N}{p+1}} \leq \min\left\{e^{\frac{4}{R(p+1)}}, e^{\frac{4}{(1-R)(p+1)}}\right\}S \leq e^{\frac{8}{p+1}}S.
\end{equation}
\end{lem}
\begin{proof}
 Let $R \in [0,1)$. Setting $x_\alpha=(1-\frac{1}{\alpha},0,...,0),$  for any $\omega \in \mathcal D(B)\backslash\{0\}$ the function $w_\alpha=\omega (\alpha(\cdot-x_\alpha))$ is supported in $B\setminus B_{R}$ for all $\alpha > \frac{2}{1-R}$. Moreover since the radial profile $V_{R}(r)$ is increasing for $r>R$ we obtain for $\alpha > \frac{2}{1-R}$
\[
S_{\alpha, R} \leq   \ddfrac{\int_B |D w_\alpha|^2 \dif x}{\left(\int_B V_{R}(x)|w_\alpha|^{p+1}\dif x\right)^{\frac{2}{p+1}}} \leq \alpha ^{2-N+\frac{2N}{p+1}} \ddfrac{\int_B |D \omega|^2 \dif x}{\left(\int_B \left(1- \frac{2}{(1-R)}\frac{1}{\alpha}\right)^\alpha |\omega|^{p+1}\dif x\right)^{\frac{2}{p+1}}}.
\]
Since $\lim_{\alpha\rightarrow\infty} \left(1- \frac{2}{(1-R)}\frac{1}{\alpha}\right)^\alpha=e^{-\frac{2}{1-R}}$, we infer that
\begin{equation}\label{neweqexp1}
\limsup_{\alpha\rightarrow \infty} S_{\alpha,R} \, \alpha^{N-2-\frac{2N}{p+1}} \leq e^{\frac{4}{(1-R)(p+1)}}S \qquad \text{for all} \ \ R \in [0,1).
\end{equation}

On the other hand, let $R \in (0,1]$. In this case we have $V_R$ is decreasing for $r < R$. Setting $x_\alpha=(\frac{1}{\alpha},0,...,0),$ for any $\omega \in \mathcal D(B)\backslash\{0\}$ the function $w_\alpha=\omega (\alpha(\cdot-x_\alpha))$ is supported in $B_{R}$ for all $\alpha > \frac{2}{R}$. Hence, for $\alpha > \frac{2}{R}$, we obtain
\[
S_{\alpha,R} \leq   \ddfrac{\int_B |D w_\alpha|^2 \dif x}{\left(\int_B V_{R}(|x|)|w_\alpha|^{p+1}\dif x\right)^{\frac{2}{p+1}}}\leq \alpha ^{2-N+\frac{2N}{p+1}} \ddfrac{\int_B |D \omega|^2 \dif x}{\left(\int_B \left(1- \frac{2}{\alpha R}\right)^\alpha |\omega|^{p+1}\dif x\right)^{\frac{2}{p+1}}}.
\]
Since $\lim_{\alpha\rightarrow\infty} \left(1- \frac{2}{\alpha R}\right)^\alpha=e^{-\frac{2}{R}}$, we infer that
\begin{equation}\label{neweqexp2}
\limsup_{\alpha\rightarrow \infty} S_{\alpha,R} \, \alpha^{N-2-\frac{2N}{p+1}} \leq e^{\frac{4}{R(p+1)}}S \qquad \text{for all} \ \ R \in (0,1].
\end{equation}
Then the conclusion follows from \eqref{neweqexp1} and \eqref{neweqexp2}.
\end{proof}

\begin{remark}
Note that for $R=1/2$ the same estimate can be proved using a convex combination $\omega_\alpha=t\omega (\alpha(\cdot-x_\alpha))+(1-t)\omega (\alpha(\cdot-x'_\alpha))$ with any $t\in[0,1]$ where $x_\alpha=(\frac{1}{\alpha},0,...,0),$ and $x'_\alpha=(1-\frac{1}{\alpha},0,...,0).$\end{remark}

\section{Relation with \eqref{eqGNN} and \eqref{classH} equations and proof of Proposition \ref{GNNH}}\label{RelationHenon}

A first consequence of the bound \eqref{eq:boundconstant} on the best constant $S_\alpha$ is the following proposition. Here we write $S_{\alpha}$ as $S_{\alpha,R}$ and $V$ as $V_R$ to emphasize their dependence on $R$.

\begin{prop} [Relation with the H\'enon equation]\label{prop:limit} 
Let $N \geq 1$. If $\alpha  >  0$ is fixed, then 
\begin{equation}\label{limitconstants}
S_{\alpha,0}=\lim_{R \to 0^+} S_{\alpha,R}.
\end{equation}
Moreover, given any positive sequence $(R_n)_{n\in \mathbb N}$ converging to zero, if $u^*_{\alpha,R_n}$ is an optimiser to $S_{\alpha,R_n}$ with $\int_B V_{R_n}(|x|)|u^*_{\alpha,R_n}|^{p+1}\dif x=1$, then $u^*_{\alpha,R_n} \to U_0$ in $H^1_0(B)$ and in $C^{2,\delta}(\overline{B}),$ up to a subsequence, where $U_0$ is an optimiser to the H\'enon quotient $S_{\alpha,0}$ with $\int_B |x|^{\alpha}|U_0|^{p+1}\dif x=1$.
\end{prop}
\begin{proof}
Pick a sequence $(R_n)_{n\in \mathbb N}$ converging to zero and consider optimisers $u^*_{\alpha,R_n}$ to $S_{\alpha,R_n}$~with $$\int_B V_{R_n}(|x|)|u^*_{\alpha,n}|^{p+1}\dif x=1.$$ 
The bound \eqref{eq:boundconstant} on $S_{\alpha,R_n}$ implies that
$$\|u^*_{\alpha,R_n}\|_{H^1_0(B)}<C(\alpha,p,N).$$	
Passing if necessary to a subsequence we have that
\begin{equation*}
\begin{array}{ll}
u^*_{\alpha,R_n} \rightharpoonup U_0 & \text{in} \,\,H^1_0(B),\\
u^*_{\alpha,R_n}\rightarrow  U_0& \text{in} \,\,L^{p+1}(B),\\
u^*_{\alpha,R_n}\rightarrow  U_0 &  \text{almost everywhere on}\,\, B.
\end{array}
	\end{equation*}
	Since $V_{R_n}\rightarrow |x|^\alpha,$ by weakly lower semicontinuity we have that
\begin{equation}\label{limit henon}
S_{\alpha,0} \!\leq \!\ddfrac{\int_B |D U_0|^2 \dif x}{\left(\int_B |x|^\alpha| U_0|^{p+1}\dif x\right)^{\frac{2}{p+1}}}\!\leq \!\liminf_{n\rightarrow \infty} \ddfrac{\int_B |D u^*_{\alpha,R_n}|^2 \dif x}{\left(\int_B V_{R_n}(|x|)|u^*_{\alpha,R_n}|^{p+1}\dif x\right)^{\frac{2}{p+1}}}\!=\!\liminf_{n\rightarrow \infty} S_{\alpha,R_n}.
\end{equation}
On the other hand, pick $U\in H^1_0(B)$ which minimises the H\'enon quotient, namely
\begin{equation} \tag{$\mathcal S_{\alpha,0}$}\label{henon}
\ddfrac{\int_B |D U|^2 \dif x}{\left(\int_B |x|^\alpha| U|^{p+1}\dif x\right)^{\frac{2}{p+1}}}=S_{\alpha,0}:=\inf_{0\neq u\in H^1_0(B)} \ddfrac{\int_B |D u|^2 \dif x}{\left(\int_B |x|^\alpha| u|^{p+1}\dif x\right)^{\frac{2}{p+1}}}.    \end{equation}
Since
$$\int_B |x|^\alpha |U|^{p+1}\dif x= \lim_{n\rightarrow \infty}\int_B V_{R_n}(|x|)|U|^{p+1}\dif x,$$ and using \eqref{limit henon}, we have
\begin{equation} \label{comparison}
\limsup_{n\rightarrow \infty} S_{\alpha,R_n}\leq \limsup_{n\rightarrow \infty} \ddfrac{\int_B |D U|^2 \dif x}{\left(\int_B V_{R_n}(|x|)| U|^{p+1}\dif x\right)^{\frac{2}{p+1}}} =S_{\alpha,0}.
\end{equation}
Therefore, from \eqref{limit henon} and \eqref{comparison}, $S_{\alpha,0}=\lim_{n \to \infty} S_{\alpha,R_n}$, $u^*_{\alpha,R_n} \to  U_0$ in $H^1_0(B)$ and since the sequence $(R_n)$ converging to zero was arbitrary we infer that  $S_{\alpha,0}=\lim_{R \to 0^+} S_{\alpha,R}$. The sub-criticality of $p$ allows to perform a classical bootstrap argument to show that the convergence is in $C^{2,\delta}(\overline{B}),$ and this concludes the proof. 
\end{proof}

\begin{proof}[Proof of Proposition \ref{GNNH}] It follows directly from Proposition \ref{prop:limit} combined with the symmetry breaking result in \cite{Smets, BW}.
\end{proof}

We end this section with a proposition which highlights the connection with the limiting problem (\ref{eqGNN}).

\begin{prop} [Relation with the equation \eqref{eqGNN}]\label{prop:limit1} 
Let $N \geq 1$. For any $\alpha>0,$ it holds that
\begin{equation}\label{limitconstants}
S_{\alpha,1}=\inf_{0\neq u\in H^1_0(B)} \ddfrac{\int_B |D u|^2 \dif x}{\left(\int_B (1-|x|)^\alpha|u|^{p+1}\dif x\right)^{\frac{2}{p+1}}} =\lim_{R \to 1^-} S_{\alpha,R}.
\end{equation}
Moreover, given any sequence $(R_n)_{n\in \mathbb N}$ converging to $1^-$, if $u^*_{\alpha,R_n}$ is an optimiser to $S_{\alpha,R_n}$ with $\int_B V_{R_n}(|x|)|u^*_{\alpha,R_n}|^{p+1}\dif x=1$, then $u^*_{\alpha,R_n} \to U_1$ in $H^1_0(B)$ and in $C^{2,\delta}(\overline{B}),$ up to a subsequence, where $U_1$ is an optimiser to the quotient $S_{\alpha,1}$ with $\int_B (1-|x|)^{\alpha}|U_1|^{p+1}\dif x=1$.
\end{prop}
\begin{proof}
The proof is identical to that of Proposition \ref{prop:limit}. We leave out the details.
\end{proof}

\section{Growth estimates for $C_{\alpha,\textrm{rad}}$}\label{main body}
The estimates which follow are essentially based on implementing Ni's inequality into a suitable rewriting of Nehari's and Pohozaev's identities associated with $u_{\alpha,\textrm{rad}}.$
\subsection*{Estimate for $\int_{\partial B} |D u_{\alpha, \textrm{rad}}|^2$ from above}
\begin{lem}\label{lemma1}
Let $N\geq 2$, $R \in [0,1]$ and $K^*(N,p)= |\mathbb S^{N-1}|^{\frac{1-p}{p+1}}\left(\frac{2(p+1)}{p-1}\right)^{\frac{2p}{p+1}}.$ Then
\begin{equation}\label{bound from above}
\int_{\partial B}|D u_{\alpha, \textrm{rad}}|^2 \dif \sigma\leq K^*(N,p)\frac{\left(C_{\alpha,\textrm{rad}}\right)^{2p/(p+1)}}{(\alpha+1)^{2/(p+1)}}.
\end{equation}
\end{lem}
\begin{proof}
Integrating (\ref{eq rescaled minimiser}) over $B,$ by radial symmetry and the divergence theorem we have:
$$|\mathbb S^{N-1}|\int_{\partial B}|D u_{\alpha, \textrm{rad}}|^2 \dif \sigma=\left(\int_BV(|x|)(u_{\alpha, \textrm{rad}})^{p}\dif x\right)^2.$$
Then, by H\"older's inequality we obtain:
\begin{equation}\label{above partial}
\int_{\partial B}|D u_{\alpha, \textrm{rad}}|^2 \dif \sigma\leq \ddfrac{1}{|\mathbb S^{N-1}|} \left(\int_BV(|x|)(u_{\alpha, \textrm{rad}})^{p+1}\dif x\right)^{2p/(p+1)}\left(\int_BV(|x|) \dif x\right)^{2/(p+1)}.
\end{equation}
By the definition (\ref{least energy}) and testing (\ref{eq rescaled minimiser}) with $u_{\alpha, \textrm{rad}}$ we infer that
\begin{equation}\label{Nehari}
\int_B V(|x|) (u_{\alpha, \textrm{rad}})^{p+1}\dif x=\frac{2(p+1)}{p-1} C_{\alpha,\textrm{rad}}.
\end{equation}
To estimate $\int_BV(|x|) \dif x,$ write
\begin{equation*}
\begin{array}{ll}
\displaystyle{\int_B} V(|x|) \dif x &= \displaystyle{\int_{B_{R}}}V(|x|) \dif x + \displaystyle{ \int_{B\setminus B_{R}}} V(|x|) \dif x \\
& = |\mathbb S^{N-1}| \left[\displaystyle{\int_0^R}  \left(1-\frac{r}{R}\right)^\alpha r^{N-1}\dif r+\displaystyle{\int^1_R} \left(1-\frac{1-r}{1-R}\right)^\alpha r^{N-1}\dif r\right] \\
& = |\mathbb S^{N-1}| \left[I+II\right].
\end{array}
\end{equation*}
Setting $s=1-r/R$ we have $$I=\int_0^1 s^\alpha \left(R(1-s)\right)^{N-1}R\dif s\leq R\int^1_0 s^\alpha \dif s= \frac{R}{\alpha+1}.$$
A similar estimate using the change of variable $s=\left(1-\frac{1-r}{1-R}\right),$ yields
 $$II \leq \int ^1_0 s^\alpha (1-R) \dif s = \frac{1-R}{\alpha+1}.$$
 It follows that
 \begin{equation}\label{V int}
 \displaystyle{\int_B} V(|x|) \dif x \leq \frac{|\mathbb S^{N-1}|}{\alpha+1} .
 \end{equation}
Finally, inserting (\ref{Nehari}) and (\ref{V int}) into (\ref{above partial}), we obtain (\ref{bound from above}). This concludes the proof.
\end{proof}

\subsection*{Estimate for $\int_{\partial B} |D u_{\alpha, \textrm{rad}}|^2$ from below}

\begin{lem}\label{lemma2}
Let $N\geq 3$ and $R\in [0,1]$. Then
\begin{multline}\label{boundbelow}
\int_{\partial B} |Du_{\alpha,\textrm{rad}}|^2\dif \sigma\geq \left(\frac{2(N+\alpha)}{p+1}-(N-2)\right)\ddfrac{2(p+1)}{p-1} C_{\alpha,\textrm{rad}} \\ - K_*(N,p) R^{\beta+1} \frac{\alpha \Gamma(\alpha)}{\Gamma(\alpha+\beta+1)} \left(C_{\alpha,\textrm{rad}}\right)^{\frac{p+1}{2}}
\end{multline} 
where
\begin{equation}\label{kappa}
K_*(N,p)=2^{\frac{p+3}{2}}(p+1)^{\frac{p-1}{2}}[(p-1)(N-2)]^{-\frac{p+1}{2}} |\mathbb S ^{N-1}|^{\frac{1-p}{2}}\Gamma(\beta+1) 
\end{equation} 
and 
\begin{equation}\label{beta}
\beta=N-1-(p+1)\frac{N-2}{2}.
\end{equation}
\end{lem}
\begin{proof}
From the Pohozaev identity 
\begin{multline}\label{Pohozaev}
\int_{\partial B} |Du_{\alpha,\textrm{rad}}|^2\dif \sigma= \left(\frac{2N}{p+1}-(N-2)\right)\int_B V(|x|) |u_{\alpha,\textrm{rad}}|^{p+1}\dif x \\+\frac{2}{p+1}\int_B V'(|x|)|x| |u_{\alpha,\textrm{rad}}|^{p+1}\dif x.
\end{multline} 
Let $R\in [0,1)$. By using the definition of $V$ it is convenient to write
\begin{multline*}
\displaystyle{\int_B} V'(|x|)|x| |u_{\alpha,\textrm{rad}}|^{p+1}\dif x =  \alpha \displaystyle{\int_{B}} V(|x|) |u_{\alpha,\textrm{rad}}|^{p+1}\dif x 
 - \alpha \displaystyle{\int_{B_R} }\left(1-\frac{|x|}{R}\right)^{\alpha-1} |u_{\alpha,\textrm{rad}}|^{p+1}\dif x\\
 + \ddfrac{\alpha R}{1-R}\displaystyle{\int_{B\setminus B_R}} \left(1-\frac{1-|x|}{1-R}\right)^{\alpha-1} |u_{\alpha,\textrm{rad}}|^{p+1}\dif x\,.
	\end{multline*}
Inserting in (\ref{Pohozaev}) and taking into account the positivity of the last integral, we obtain 
\begin{multline}\label{Pohozaev2}
\int_{\partial B} | Du_{\alpha,\textrm{rad}}|^2\dif \sigma\geq \left(\frac{2(N+\alpha)}{p+1}-(N-2)\right)\int_B V(|x|) |u_{\alpha,\textrm{rad}}|^{p+1}\dif x\\-\frac{2\alpha}{p+1}\displaystyle{\int_{B_R} }\left(1-\frac{|x|}{R}\right)^{\alpha-1} |u_{\alpha,\textrm{rad}}|^{p+1}\dif x.
\end{multline} 
By Ni's inequality \cite[eq. (4)]{Ni}
\begin{equation}\label{Niineq} |u_{\alpha,\textrm{rad}}(x)| \leq \frac{1}{\left(|\mathbb S ^{N-1}|(N-2)\right)^{1/2}} \frac{ \| D u_{\alpha,\textrm{rad}} \|_{ L^2(B)}} {|x|^{(N-2)/2}} ,
\end{equation}
we estimate 
\begin{equation*}
\begin{array}{ll}
\displaystyle{\int_{B_R} }\left(1-\frac{|x|}{R}\right)^{\alpha-1} |u_{\alpha,\textrm{rad}}|^{p+1}\dif x \leq  & |\mathbb S^{N-1}| \left(\frac{1}{|\mathbb S ^{N-1}|(N-2)}\right)^{\frac{p+1}{2}}\left[C_{\alpha,\textrm{rad}} \frac{2(p+1)}{p-1}\right]^{\frac{p+1}{2}}\\
&\times \displaystyle{\int^R_0} \left( 1-\frac{r}{R}\right)^{\alpha-1}r^{N-1-(p+1)\frac{N-2}{2}}\dif r.
\end{array}
	\end{equation*}
Set $$I_R= \displaystyle{\int^R_0} \left( 1-\frac{r}{R}\right)^{\alpha-1}r^{N-1-(p+1)\frac{N-2}{2}}\dif r.$$
By the change of variable $s=1-r/R$ and by the definition of $\beta$ we can express $$I_R=R^{\beta+1}\displaystyle{\int^1_0} s^{\alpha-1}(1-s)^{\beta}\dif s=R^{\beta+1}\ddfrac{\Gamma(\alpha)\Gamma(\beta+1)}{\Gamma(\alpha+\beta+1)},$$ where we have used the well-known expression of Euler's integral of the first kind in terms of Beta function $B(\alpha,\beta+1).$
Coming back to (\ref{Pohozaev2}) this immediately yields the desired estimate for $R\in [0,1)$. Finally, observing that all the calculations after \eqref{Pohozaev2} also holds in case $R=1$, we conclude the proof. 
\end{proof} 

\begin{lem}\label{lemma2'}
Let $N=2$ and $R\in [0,1]$. Then
\begin{equation}\label{boundbelow}
\int_{\partial B} |Du_{\alpha,\textrm{rad}}|^2\dif \sigma\geq \frac{4(2+\alpha)}{p-1} C_{\alpha,\textrm{rad}}- K_*(2,p) R^{\beta+1} \frac{\alpha \Gamma(\alpha)}{\Gamma(\alpha+\beta+1)} \left(C_{\alpha,\textrm{rad}}\right)^{\frac{p+1}{2}}
\end{equation} 
where
\begin{equation}\label{kappa}
K_*(2,p)= 4  \pi^{\frac{1-p}{2}} (p+1)^{\frac{p-1}{2}}(p-1)^{-\frac{p+1}{2}} \Gamma(\beta+1) c_\varepsilon^{p+1},
\end{equation} 
\begin{equation}\label{beta}
\beta=1-(p+1)\varepsilon,
\end{equation}
and $c_\varepsilon=\sup_{r\in(0,1)}\left(r^\varepsilon\Big|\ln r\Big|^{1/2}\right),$ with $0<\varepsilon<\frac{2}{p+1}.$
\end{lem}
\begin{proof}
The only change with respect to the proof of Lemma \ref{lemma2} is in the use of Ni's inequality, which we replace by an estimate valid for $N=2$, namely
\begin{equation*}
|u_{\alpha,\textrm{rad}}(x)| \leq \left(\frac{ |\ln |x|| } {|\mathbb S ^{1}|}\right)^{1/2}\| D u_{\alpha,\textrm{rad}}\|_{ L^2(B)} ,
\end{equation*}
see \cite[Lemma 2.5]{BonheureSerraTarallo}. In view of the definition of $c_\varepsilon$ we write the above inequality as 
\begin{equation}\label{nonlog}
|u_{\alpha,\textrm{rad}}(x)| \leq \frac{c_\varepsilon}{(2\pi)^{1/2}} \frac{ \| D u_{\alpha,\textrm{rad}} \|_{ L^2(\R^2)}} { |x|^\varepsilon}.
\end{equation}
Hence
\begin{equation*}
\displaystyle{\int_{B_R} }\!\!\left(1-\frac{|x|}{R}\right)^{\alpha-1}\!\! |u_{\alpha,\textrm{rad}}|^{p+1}\dif x \leq   c_\varepsilon^{p+1}(2\pi)^{\frac{1-p}{2}}\!\!\left[C_{\alpha,\textrm{rad}} \frac{2(p+1)}{p-1}\right]^{\frac{p+1}{2}} \! \!\!
 \displaystyle{\int^R_0}\!\! \left( 1-\frac{r}{R}\right)^{\alpha-1}\!\!r^{1-(p+1)\varepsilon}\dif r.
	\end{equation*}
Set $$I_R= \displaystyle{\int^R_0} \left( 1-\frac{r}{R}\right)^{\alpha-1}r^{1-(p+1)\varepsilon} \dif r.$$
By the change of variable $s=1-r/R$ and setting $\beta=1-(p+1)\varepsilon$ we can express $$I_R=R^{\beta+1}\displaystyle{\int^1_0} s^{\alpha-1}(1-s)^{\beta}\dif s=R^{\beta+1}\ddfrac{\Gamma(\alpha)\Gamma(\beta+1)}{\Gamma(\alpha+\beta+1)}.$$ Using this estimate in (\ref{Pohozaev2}) gives immediately the statement for $N=2$. And this concludes the proof.
\end{proof}
\subsection*{Estimate for $C_{\alpha,\textrm{rad}}$}

\begin{lem}\label{lemma3}
Let $N\geq 2$ and $R\in [0,1]$. Then
\begin{equation}\label{radial short}
\left(\frac{2(N+\alpha)}{p+1}-(N-2)\right)\ddfrac{2(p+1)}{p-1} C_{\alpha,\textrm{rad}}\leq A_\alpha(R)+B_\alpha
\end{equation} 
where
\begin{equation}\label{A}
A_\alpha(R)= K_*(N,p) R^{\beta+1} \ddfrac{\alpha \Gamma(\alpha)}{\Gamma(\alpha+\beta+1)} \left(C_{\alpha,\textrm{rad}}\right)^{\frac{p+1}{2}} 
\end{equation} 
and
\begin{equation}\label{B}
B_\alpha= K^*(N,p)\ddfrac{\left(C_{\alpha,\textrm{rad}}\right)^{2p/(p+1)}}{(\alpha+1)^{2/(p+1)}}.
\end{equation} 
\end{lem}
\begin{proof}

Combining Lemma \ref{lemma1} with Lemma \ref{lemma2} (for $N\geq3$) and Lemma \ref{lemma2'} (for $N=2$), we immediately obtain the estimate.
\end{proof}

\subsection*{Estimate for $S_{\alpha,\textrm{rad}}$ when $R=1$ and $p=1$} 

This estimate will be useful in the proof of Proposition \ref{R1}.
In the case $p=1$ our PDE becomes an eigenvalue problem, and since the scaling $ u_{\alpha,\textrm{rad}}=\left(S_{\alpha,\textrm{rad}}\right)^{\frac{1}{p-1}}u^*_{\alpha,\textrm{rad}}$ is not defined, we provide directly an estimate for $S_{\alpha,\textrm{rad}}.$

\begin{lem}\label{p1} Let $N\geq 3$, $R=1$ and $p=1$. There exists $C>0$ such that
\begin{equation} \label{onedimGround}
\liminf_{\alpha\rightarrow \infty}\alpha^{-2}S_{\alpha,\textrm{rad}} \geq C. 
\end{equation}
\end{lem}
\begin{proof}
Since $R=1$, by \cite{Gidas}, $S_{\alpha,\textrm{rad}}=S_\alpha$. Pick a radial optimiser $u_\alpha \in H^{1}_0(B)$ for $S_{\alpha,\textrm{rad}}.$
By using Ni's inequality (\ref{Niineq}) we have
$$\displaystyle{\int_{0}^{1} } (1-r)^\alpha |u_{\alpha}(r)|^2 r^{N-1} \dif r \leq C' \|Du_\alpha\|^2_{L^2(B)}\displaystyle{\int^1_0}(1- r)^{\alpha} r \dif r =C' \ddfrac{\|Du_\alpha\|^2_{L^2(B)}}{\alpha^2+3\alpha+2}.$$
This immediately implies that there exists $C>0$ such that
	\begin{equation*}
\liminf_{\alpha\rightarrow \infty}\alpha^{-2}\ddfrac{\int_B |D u_\alpha|^2 \dif x}{\int_B(1-|x|)^\alpha |u_\alpha|^2\dif x} \geq C,
\end{equation*}
and this concludes the proof.
\end{proof}

\section{Quantitative symmetry breaking: Proofs of Theorems \ref{symmetry breaking theorem} and \ref{symmetry breaking theorem2} } \label{symmetry breaking proof}

\begin{proof}[Proof of Theorem \ref{symmetry breaking theorem}]
Let $A_{\alpha}(R)$ and $B_{\alpha}$ as in Lemma \ref{lemma3} and set $$L(R):=\limsup_{\alpha \rightarrow \infty}\ddfrac{A_\alpha(R)}{B_\alpha},$$
to emphasize its dependence on $R$. We distinguish the cases $L(R)$ finite and $L(R)=+\infty.$ 

\medskip

\noindent {\it Case 1:~$L(R)\in [0,+\infty).$ } 
From Lemma \ref{lemma3} and using the relation
\begin{equation}\label{eq:repetS}
S_{\alpha,\textrm{rad}}=\left(\frac{2(p+1)}{p-1}\right)^{\frac{p-1}{p+1}} \left(C_{\alpha, \textrm{rad}}\right)^{\frac{p-1}{p+1}}
\end{equation}
we obtain a lower bound for $S_{\alpha,\textrm{rad}}$ of the form 

$$S_{\alpha,\textrm{rad}} \geq C(N,p,R)\, \alpha^{1+\frac{2}{p+1}}+o(1),\,\qquad\alpha\rightarrow +\infty.$$
Since $1<p$, by this estimate and Lemma \ref{Salpha} 
it follows that there exists $\alpha=\alpha(R)$  such that $$S_\alpha <S_{\alpha,\textrm{rad}}, \quad \textrm{for all} \,\,\alpha>\alpha(R).$$

\noindent {\it Case 2:~$L(R)=+\infty.$ } Here we assume $N \geq 3$ and write the estimate (\ref{radial short}) in the form 
\begin{equation}\label{bad}
\left(\frac{2(N+\alpha)}{p+1}-(N-2)\right)\ddfrac{2(p+1)}{p-1} C_{\alpha,\textrm{rad}}\leq A_\alpha\left[1+\ddfrac{B_\alpha}{A_\alpha}\right].
\end{equation} 
Using again \eqref{eq:repetS}, we infer that
\begin{equation}\label{badd}
1+\ddfrac{1}{\alpha}\left[\ddfrac{2N}{p+1}-(N-2)\right] \ddfrac{p+1}{2}\leq   K(N,p) R^{\beta+1} \ddfrac{ \Gamma(\alpha)}{\Gamma(\alpha+\beta+1)} \left(S_{\alpha,\textrm{rad}}\right)^{\frac{p+1}{2}} \left[1+\ddfrac{B_\alpha}{A_\alpha}\right]
\end{equation} 
where
$$K(N,p)=\ddfrac{p+1}{2}\left(\frac{p-1}{2(p+1)} \right)^{\frac{p+1}{2}} K_*(N,p), $$ namely 
the constant given in \eqref{Kappa}. Classical asymptotic estimates on the Gamma function, see e.g. \cite{Tricomi}, yield 
\begin{equation*}\label{asymptotic}
\ddfrac{ \Gamma(\alpha)}{\Gamma(\alpha+\beta+1)}\sim \alpha^{-\beta-1} \ \textrm{ as } \ \alpha \rightarrow \infty, 
\end{equation*} 
that is
\begin{equation*}\label{asymptotic2}
\ddfrac{ \Gamma(\alpha)}{\Gamma(\alpha+\beta+1)}\sim \alpha^{\frac{p+1}{2}(N-2)-N}  \ \textrm{ as }\ \alpha \rightarrow \infty.
\end{equation*}  
From this and (\ref{badd}) we finally obtain
\begin{equation}\label{asymptotic3}
\liminf_{\alpha\rightarrow \infty} S_{\alpha,\textrm{rad}} \alpha^{N-2-\frac{2N}{p+1}}\geq K(N,p)^{-\frac{2}{p+1}} R^{N-2-\frac{2N}{p+1}} . 
\end{equation}  
Since by Lemma \ref{Salpha} we have
\begin{equation*}
\limsup_{\alpha\rightarrow \infty} S_\alpha \, \alpha^{N-2-\frac{2N}{p+1}} \leq \min\left\{e^{\frac{4}{R(p+1)}}, e^{\frac{4}{(1-R)(p+1)}}\right\}S,
\end{equation*}
and since the hypotheses on $R$ and the sub-criticality of $p$ imply the strict inequality $$\min\left\{e^{\frac{4}{R(p+1)}}, e^{\frac{4}{(1-R)(p+1)}}\right\}S- K(N,p)^{-\frac{2}{p+1}} R^{N-2-\frac{2N}{p+1}}<0,$$
it follows that 
$$\limsup_{\alpha\rightarrow \infty} (S_\alpha-S_{\alpha,\textrm{rad}})<0$$
and this concludes the proof.
\end{proof}

\begin{remark}\label{rem:N=2} If $R=0$, then $A_\alpha =0$. So we are in {\it Case $1$} above, whose arguments works for $N\geq2$. Then we recover, with a slightly different method, the symmetry breaking result for the H\'enon equation \cite[Theorems 3.1 and 4.2]{Smets}, with a unified proof that works for $N\geq 2$. The idea of using the Pohozaev identity to obtain energy bounds in the radial framework has been inspired by \cite{Denis-Ederson-Miguel-JFA}.
\end{remark}

\begin{remark}\label{remarkN2}
We observe that the above argument for $N=2$ in {\it Case $2$} does not imply $S_{\alpha, rad} > S_{\alpha}$ for large values of $\alpha$. In fact carrying out the proof in {\it Case $2$}, using the expression of $\beta=1-(p+1)\varepsilon$  provided by Lemma \ref{lemma2'},
one finds 
$$
\liminf_{\alpha\rightarrow \infty} S_{\alpha,\textrm{rad}} \alpha^{2\varepsilon-\frac{4}{p+1}}\geq C(N,p, R,\varepsilon),
$$
which does not match with the asymptotic estimate provided by Lemma \ref{Salpha}
$$
\limsup_{\alpha\rightarrow \infty} S_\alpha \, \alpha^{-\frac{4}{p+1}} \leq \min\left\{e^{\frac{4}{R(p+1)}}, e^{\frac{4}{(1-R)(p+1)}}\right\}S.
$$
\end{remark}

\begin{proof}[Proof of Theorem \ref{symmetry breaking theorem2}]
Let $A_{\alpha}$ and $B_{\alpha}$ as in Lemma \ref{lemma3} having replaced $R$ with $R(\alpha)$ and set $$L:=\limsup_{\alpha \rightarrow \infty}\ddfrac{A_\alpha}{B_\alpha}.$$
As in the proof of Theorem \ref{symmetry breaking theorem}, symmetry breaking occurs if $L$ is finite.
So let us consider the case $L=+\infty.$  In this case by using Lemma \ref{lemma2'} we have for all $\varepsilon\in\big(0,\frac{2}{p+1}\big)$
\begin{equation}\label{badd2}
1+\ddfrac{2}{\alpha}\leq   C(p,\varepsilon) R(\alpha)^{2-\varepsilon(p+1)} \ddfrac{ \Gamma(\alpha)}{\Gamma(\alpha+2-\varepsilon(p+1))} \left(S_{\alpha,\textrm{rad}}\right)^{\frac{p+1}{2}} \left[1+\ddfrac{B_\alpha}{A_\alpha}\right].
\end{equation} 
Using the asymptotic estimates on the Gamma function \cite{Tricomi}, we obtain that for all $\varepsilon\in\big(0,\frac{2}{p+1}\big)$ it holds that
\begin{equation}\label{asymptoticN2}
\liminf_{\alpha\rightarrow \infty} S_{\alpha,\textrm{rad}}\Big(\frac{R(\alpha)}{\alpha}\Big)^{\frac{4}{p+1}-2\varepsilon}\geq C(p,\varepsilon)>0
\end{equation} 
for some constant $C(p,\varepsilon).$
With the particular choice
$$\varepsilon=\dfrac{4\delta}{(p+1)(2\delta+3)}\in\big(0,\frac{2}{p+1}\big),$$ since

\begin{equation*}\label{asymptoticN22}
\Big(\frac{R(\alpha)}{\alpha}\Big)^{\frac{4}{p+1}-2\varepsilon}=\alpha^{-\frac{4}{p+1}-\varepsilon} \Big(\alpha^\delta R(\alpha)\Big)^{\frac{4}{p+1}-2\varepsilon}\end{equation*} 
we obtain from (\ref{asymptoticN2}) and by the hypothesis on $R(\cdot)$ that 

$$
\liminf_{\alpha\rightarrow \infty} S_{\alpha,\textrm{rad}}\alpha^{-\frac{4}{p+1}-\varepsilon}\geq C'(p,\varepsilon)>0.
$$
This means that the growth of $S_{\alpha,\textrm{rad}}$ is faster than that of $S_\alpha$ provided by Lemma \ref{Salpha} and this concludes the proof.
\end{proof}

\section{Necessary conditions for concentration of groundstates and proofs of Theorem \ref{nec} and of Proposition \ref{R1}}

Throughout this section we consider $N\geq 3$. Before proving Proposition \ref{R1} it is worth making the following remark, which is of independent interest. 
\begin{remark}\label{lowerSbound}
From the proof of Theorem \ref{symmetry breaking theorem} it is clear that when $R=1$ it holds necessarily that $L=+\infty,$  as $L\in [0,+\infty)$ would imply symmetry breaking, a contradiction by the symmetry result of Gidas-Ni-Nirenberg \cite{Gidas}. Therefore 
from (\ref{asymptotic3}) we obtain
\begin{equation}\label{useful R1}
\liminf_{\alpha\rightarrow \infty} S_{\alpha,\textrm{rad}} \alpha^{N-2-\frac{2N}{p+1}}\geq K(N,p)^{-\frac{2}{p+1}}.
\end{equation}
 Since $S_\alpha=S_{\alpha,\textrm{rad}}$ in this case, 
the above estimate combined with Lemma \ref{Salpha} yields the lower bound on $S$ provided in the Remark \ref{lowerCorollary}.
\end{remark} 
\begin{proof}[Proof of Proposition \ref{R1}]
Following the Remark \ref{lowerSbound}, the Gidas-Ni-Nirenberg result \cite{Gidas} implies that $S_\alpha=S_{\alpha,\textrm{rad}}$ and the double sided growth estimates for $S_\alpha$ holds, namely
(\ref{useful R1}) in addition to Lemma \ref{Salpha}. Picking a groundstate $u_\alpha,$  by the Nehary identity it follows 
$$
\displaystyle{\int_{B} |D u_\alpha|^2\dif x}=\left(S_{\alpha}\right)^{\frac{p+1}{p-1}}.
$$
Hence, setting $\displaystyle{v_\alpha(x):=\alpha^{\frac{2}{1-p}} u_\alpha \left(\frac{x}{\alpha}\right),}$
we obtain
\begin{equation}\label{carab}
C_1\leq \int_{B_\alpha} |D v_\alpha|^2\dif x\leq C_2
\end{equation}
for some constant $C_1,C_2>0.$ Using \eqref{onedimGround} from Lemma \ref{p1} we have,  as in \cite{CaoPeng} p. 475, 

$$\displaystyle{\int_{B} } (1-|x|)^\alpha |u_{\alpha}|^2  \dif x \leq C \alpha^{-2} \displaystyle{\int_{B} }  |D u_{\alpha}|^2  \dif x.$$
This implies for some $C>0$ that
$$\displaystyle{\int_{B_\alpha} } \left(1-\Big|\frac{x}{\alpha}\Big|\right)^{\alpha} |v_{\alpha}|^2  \dif x \leq C \displaystyle{\int_{B_{\alpha}} }  |D v_{\alpha}|^2  \dif x\leq C'_2,$$
and finally, by Nehari's identity and (\ref{carab}) we obtain
\begin{equation}\label{Z1}
C_1 \leq \max_{\overline{B}_{\alpha}} |v_{\alpha}|^{p-1} \displaystyle{\int_{B_{\alpha}} } \left(1-\Big|\frac{x}{\alpha}\Big|\right)^{\alpha} |v_{\alpha}|^2  \dif x\leq C'_2\max_{\overline{B}_{\alpha}} |v_{\alpha}|^{p-1} .
\end{equation}
On the other hand, since $v_{\alpha}\in H^1_0(B_{\alpha})$ satisfies $-\Delta v_{\alpha} \leq v^p_{\alpha},$
by classical Moser's iteration we can show that the uniform bound 
\begin{equation}\label{Z2}
\|v_{\alpha}\|_\infty\leq C
\end{equation}
holds. The conclusion follows now immediately by (\ref{Z1}) and (\ref{Z2}).
\end{proof}

For all $R\in[0,1]$ the norm  of groundstate solutions, $\|u_\alpha\|_\infty,$ blows up as $\alpha\rightarrow \infty.$ However, when $R\neq 0,1$, we are able to show only that this occurs at a growth rate which is slower than that required by Theorem \ref{nec}.

\begin{prop} [ $\|u_\alpha\|_\infty$ blows up]\label{blowup prop}
Let $N\geq 1,$ $R\in [0,1]$ and for any $\alpha>0$, denote by $u_\alpha$ a positive groundstate solution of \eqref{eq}. Then $\|u_\alpha\|_{\infty}\rightarrow +\infty$ as $\alpha \rightarrow +\infty$.
\end{prop}
\begin{proof}
Set $\beta_\alpha=\|u_\alpha\|_{\infty}$  and recall that by \eqref{V int}, for $N\geq 2$ we have
$$ \displaystyle{\int_B} V(|x|) \dif x \leq \frac{|\mathbb S^{N-1}|}{\alpha+1}.$$ 
For $N=1$ note that $$ \displaystyle{\int_{-1}^1} V(|x|) \dif x = \frac{2}{\alpha+1} = \frac{|\mathbb S^{0}|}{\alpha+1}.$$ 
\noindent 
Combining these and Nehari's identity we infer that $$
S^{\frac{p+1}{p-1}}\leq \left(S_{\alpha}\right)^{\frac{p+1}{p-1}}=\displaystyle{\int_{B} |D u_\alpha|^2\dif x}=\displaystyle{\int_{B} V(|x|)|u_\alpha|^{p+1}\dif x}\leq  \dfrac{C_N}{\alpha+1}\beta_\alpha^{p+1},$$
which concludes the proof.
\end{proof}

\begin{proof} [Proof of Theorem \ref{nec}] Let $(\alpha_m)$ be a sequence with $\alpha_m \to +\infty$ and, for each $m$, let $u_m$ be a positive groundstate solution of \eqref{eq}. Set $\beta_{m}:=\max_{x\in \overline{B}} u_{m}=u_{m}(x_{m})$ and
$$v_m(x):=\frac{1}{\beta_m} u_m (x_m+\beta_m^{\frac{1-p}{2}}x).$$
Then $v_m\in H^{1}_0(\Omega_m)$, where $\Omega_m:=\beta_m^{\frac{p-1}{2}}\left(B(0,1)-x_m\right),$ and
\begin{equation}\label{scaled ground}
-\Delta v_m=V_{\alpha_m}\left(|x_m+\beta_m^{\frac{1-p}{2}}x|\right)v^p_m, \qquad v_m>0, \qquad |v_m|\leq 1.
\end{equation}

Set $d:=\liminf_{m\rightarrow + \infty} \beta_m^{\frac{p-1}{2}}\textrm{dist}\left(x_m,\partial B\right).$ If $d=0,$ necessarily concentration occurs at the boundary. So below we analyse the case $d>0.$

Set $$\gamma(x):=\liminf_{m\rightarrow + \infty}V_{\alpha_m}(|x_m+\beta_m^{\frac{1-p}{2}}x|).$$ 
Then $\gamma$ is defined on a halfspace $H$ or on $\R^N$ according to whether $d\in(0,\infty)$ or $d=+\infty.$  

If $\gamma$ is not constant, then $(x_m)_{m\in\mathbb N}$ accumulates at a point on $\partial B \cup \{0\},$ as otherwise  $\gamma=0.$ 

Therefore, we consider the case  $\gamma(x)$ is a constant and show that in fact this case does not occur by showing that we get a contradiction. We have sub-cases according to the different values of $d.$ 
If $d= +\infty$ we see that as $m$ gets  large enough, for every fixed $\rho>0$ any ball $B_\rho$ is properly contained in $\Omega_m.$ Classical elliptic estimates and bootstrap show that $v_m$ is uniformly bounded in $C^{2,\delta}_{\textrm{loc}}(\R^N)$ and, as a consequence, a suitable subsequence converges in $C^{2,\delta'}_{\textrm{loc}}(\R^N)$ to some $v$ satisfying 
\begin{equation}\label{whole eq}
-\Delta v=\gamma v^p \ \ \textrm{on}\,\,\R^N, \qquad v>0 \ \ \textrm{on}\,\,\R^N, \qquad v(0)= 1.
\end{equation}
\noindent Similarly, if $d\in (0,\infty)$ we find that some nonnegative $v\in C^{2,\delta}_{\textrm{loc}}(\overline{H})$ exists, defined on a halfspace $H$ and satisfying
\begin{equation}\label{half eq}
-\Delta v=\gamma v^p \ \  \textrm{on}\,\,H, \qquad \qquad v>0 \ \ \textrm{on}\,\,H,\qquad v|_{\partial H}=0, \qquad v(0)= 1.
\end{equation}
In the case $\gamma >0$ we get a contradiction, since the classical result of Gidas-Spruck \cite{Gidas0} says that both equations (\ref{whole eq}) and (\ref{half eq}) have no solution. 

If $\gamma=0$ and $d\in (0,\infty)$ again we have contradiction, at this time using the average properties of harmonic functions.

The only possibility left is that $d=+\infty$ and $\gamma=0.$ In this case by classical classification results on positive harmonic functions we conclude that $v\equiv \textrm{constant},$ and therefore $v\equiv 1$ necessarily.
To see that this is a contradiction note that as in (\ref{relation}) and by Nehari's indentity we have
$$
\displaystyle{\int_{B} |D u_m|^2\dif x}=\left(S_{\alpha_{m}}\right)^{\frac{p+1}{p-1}}.
$$
Hence using Lemma \ref{Salpha} to estimate $S_{\alpha_m}$ we have 
\begin{equation*}
\begin{array}{ll}
\displaystyle{\int_{\R^N} \!|D v_m|^2\dif x}\!\!\!\!&=\beta_m^{{\frac{p(N-2)-N-2}{2}}}\displaystyle{\int_{B} |D u_m|^2\dif x} \\ \\
&\leq C(N,p,R) \,\beta_m^{\frac{p(N-2)-N-2}{2}}\alpha_m^{\frac{2N+(2-N)(p+1)}{p-1}} \!= C(N,p,R)\left(\alpha_m^{\frac{2}{p-1}}\beta_m^{-1}\right)^{\frac{N+2-p(N-2)}{2}}.
\end{array}
	\end{equation*}
Then, from \eqref{necessarygrowth}, it follows that up to subsequences, $(v_m)_{m\in \mathbb N}$ is bounded in $D^{1,2}(\R^N)$ and $v_m \rightharpoonup v \in D^{1,2}(\R^N).$ 
As $v$ is harmonic, by Sobolev's inequality it follows that $v\equiv 0,$ a contradiction. This concludes the proof.
\end{proof}

\section{Final remarks and related problems}
There are a number of questions related to our work which we believe it is worth studying in future projects. Here below we list some of them.

\begin{itemize}
\item  How close to $\partial B$ the zeros of $V,$ i.e. $|x|=R,$ are allowed to be in order for symmetry breaking to occur?
\item To prove (or disprove) that \eqref{necessarygrowth} holds for every $R \in (0,1)$.
\item To study the limiting profile of groundstates when $\alpha\rightarrow \infty.$
\item Let $R \in (0,1)$. May groundstate solutions concentrate simultaneously at the origin and at $(1,0, \ldots, 0)$?
\item  Is it possible to obtain similar results replacing $V$ with an oscillating function, e.g.  $|\cos(k\pi|x|)|^\alpha$?
\end{itemize}

\section*{Acknowledgements}
Carlo Mercuri was partially supported by Santander Mobility Grants; he would like to thank Cristina Montero Mudresh from the International Development Office - Swansea University, for the kind advice, and the Instituto de Ci\^encias Matem\'aticas e de Computa\c{c}\~ao - USP S\~ao Carlos, for the warm hospitality when this project has started.  Ederson Moreira dos Santos was partially supported by CNPq grant 307358/2015-1 and FAPESP grant 2015/17096-6. We thank both anonymous Referees for the careful reading of the first version of this paper.


\begin{bibdiv}

\begin{biblist}

\bib{Anello}{article}{
   author={G. Anello},
    author={F. Faraci},
    author={A. Iannizzotto},
           title={On a problem of Huang concerning best constants in Sobolev embeddings},
   journal={ Ann. Mat. Pura Appl.},
   volume={4},
   date={2015},
   number={3},
   pages={767--779}
}


\bib{batt-f-horst}{article}{
author={J.~Batt}, author={W.~Faltenbacher}, author={E.~Horst},
title={Stationary spherically symmetric models in stellar dynamics},
journal={Arch. Rational Mech. Anal.}, 
volume={93},
date={1986},
number={2},
pages={159--183}
}

\bib{batt}{article}{
author={J{{\"u}}rgen Batt},
title={Global symmetric solutions of the initial value problem of stellar
  dynamics},
 journal={J. Differential Equations}, 
 volume={25},
 date={1977},
 number={3},
 pages={342--364}
 }

\bib{batt-li}{article}{
author={J{{\"u}}rgen Batt},
author= {Yi~Li},
title={The positive solutions of the {M}atukuma equation and the problem of
  finite radius and finite mass},
journal={Arch. Ration. Mech. Anal.}, 
volume={198},
date={2010},
number={2},
pages={613--675}
}


\bib{BWW}{article}{
   author={T. Bartsch},
    author={T. Weth},
    author={M. Willem}
       title={Partial symmetry of least energy nodal solutions to some variational problems},
   journal={J. Anal. Math.},
   volume={96},
   date={2005},
   pages={1--18}
}






\bib{Denis-Ederson-Miguel-JFA}{article}{
   author={D. Bonheure},
    author={E. Moreira dos Santos},
    author={M. Ramos},
    
       title={Symmetry and symmetry breaking for ground state solutions of some strongly coupled elliptic systems},
   journal={Journal of Functional Analysis},
   volume={264},
   date={2012},
   number={1},
   pages={62--96}
}

\bib{BonheureSerraTarallo}{article}{
author={D. Bonheure},
author={E. Serra},
author={M. Tarallo},

title={Symmetry of extremal functions in Moser-Trudinger inequalities and a H\'enon type problem in dimension two},
journal={Advances in Differential Equations},
volume={13},
date={2008}, 
number={1--2}, 
pages={105--138}
}

\bib{BW}{article}{
   author={J. Byeon},
    author={Z.Q. Wang},
  
    
       title={On the H\'enon equation: asymptotic profile of ground states. I,}
   journal={Ann. Inst. H. Poincar\'e Anal. Non Lin\'eaire},
   volume={23},
   date={2006},
   number={6},
   pages={803--828}
}



\bib{CaoPeng}{article}{
   author={D. Cao},
    author={S. Peng},
    author={S. Yan},
    
       title={Asymptotic behaviour of ground state solutions for the H\'enon equation},
   journal={IMA J. Appl. Math.},
   volume={74},
   date={2009},
   number={3},
   pages={468--480}
}




\bib{Stellar}{book}{
   author={S. Chandrasekhar},
   title={An introduction to the study of stellar structure},
   publisher={Dover Publications, Inc., New York, N. Y.},
   date={1957}
}


\bib{Gidas0}{article}{
   author={B. Gidas},
    author={J. Spruck},
       title={A priori bounds for positive solutions of nonlinear elliptic equations},
   journal={ Comm. PDE},
   volume={6},
   date={1981},
   pages={883--901}
}


\bib{Gidas}{article}{
   author={B. Gidas},
    author={W.-M. Ni},
    author={L. Nirenberg}
       title={Symmetry and related properties via the Maximum Principle },
   journal={ Communications in Mathematical Physics},
   volume={68},
   date={1979},
   pages={209--243}
}

\bib{Haw}{book}{
   author={S. W. Hawking},
   author={G. F. Ellis}
   title={The large scale structure of space-time},
   publisher={Cambridge Monographs on Mathematical Physics, No. 1. Cambridge University Press, London-New York,}
   date={1973}
}

\bib{Henon}{article}{
   author={M. H\'enon},
       title={Numerical experiments on the stability of spherical stellar systems,}
   journal={Astronomy and Astrophysics },
   volume={24},
   date={1973},
   pages={229--238}
}


\bib{li}{article}{
author={Yi~Li},
title={On the positive solutions of the {M}atukuma equation},
journal={Duke Math. J.},
volume={70},
date={1993},
number={3},
pages={575--589}
}

\bib{li-santanilla}{article}{
author={Yi~Li},
author={Jairo Santanilla},
title={Existence and nonexistence of positive singular solutions for
  semilinear elliptic problems with applications in astrophysics},
  journal={Differential Integral Equations}, 
  volume={8},
  date={1995},
  number={6},
  pages={1369--1383}
  }

\bib{existBHGC}{article}{
author={T.J. Maccarone},
author={A.~Kundu},
author={S.E. Zepf},
author={K.L. Rhode},
title={A black hole in a globular cluster},
journal={Nature},
volume={445},
date={2007},
pages={183--185}
}


\bib{pacella-indiana}{article}{
author={E.~Moreira dos Santos},
author={F. Pacella},
title={H\'{e}non-type equations and concentration on spheres},
journal={Indiana Univ. Math. J.},
volume={65},
date={2016},
number={1},
pages={273--306}
}

\bib{Ni}{article}{
   author={W.M. Ni},
       title={A nonlinear Dirichlet problem on the unit ball and its applications },
   journal={Indiana Univ. Math. J. },
   volume={31},
   date={1982},
   number={6},
   pages={801--807}
}

\bib{peebles2}{article}{
author={P.I.E. Peebles},
title={Black holes are where you find them},
journal={General Relativity and Gravitation},
volume={3},
date={1972},
pages={63--82}
}

\bib{peebles1}{article}{
author={P.I.E. Peebles},
title={Star distribution near a collapsed object},
journal={Astrophysical Journal},
volume={178},
date={1972},
pages={371--376}
}





\bib{Smets}{article}{
   author={D. Smets},
    author={J. Su},
    author={M. Willem}
       title={Non-radial ground states for the H\'enon equation },
   journal={ Communications in Contemporary Mathematics},
   volume={4},
   date={2002},
   number={3},
   pages={467--480}
}

\bib{Strauss}{article}{
   author={W. Strauss},
       title={Existence of solitary waves in higher dimensions },
   journal={ Comm. Math. Phys},
   volume={55},
   date={1977},
   number={2},
   pages={149--162}
}


\bib{Talenti }{article} {
			AUTHOR = {G. Talenti},
			TITLE = {Best constant in {S}obolev inequality},
			JOURNAL = {Ann. Mat. Pura Appl. (4)},
			FJOURNAL = {Annali di Matematica Pura ed Applicata. Serie Quarta},
			VOLUME = {110},
			YEAR = {1976},
			PAGES = {353--372},
		}


\bib{Tricomi}{article}{
   author={F. Tricomi},
    author={A. Erd\'eli},
       title={The asymptotic expansions of a ratio of Gamma functions },
   journal={Pacific Journal of Mathematics },
   volume={1},
   date={1951},
   number={1},
   pages={133--142}
}


\bib{Wald}{book}{
   author={R. Wald},
   title={General Relativity},
   publisher={University of Chicago Press, Chicago,},
   date={1984}
}


\bib{Pic-blackhole}{article}{
title={First M87 Event Horizon Telescope Results. I. The Shadow of the Supermassive Black Hole },
journal={The Astrophysical Journal Letters },
   volume={875},
   date={2019},
   number={1},
   pages={L1}
   }



\end{biblist}

\end{bibdiv}
\end{document}